\documentclass[a4paper,12pt]{amsart}
\usepackage{amssymb, amsmath}
\usepackage{natbib}
\usepackage{lscape,comment}
\usepackage{color}
\usepackage{dsfont}
\usepackage{mathrsfs}
\usepackage{bm}
\usepackage{graphicx}
\usepackage{wrapfig}
\usepackage{caption}
\usepackage{tikz}
\usepackage{pgfplots}
\usetikzlibrary{arrows,decorations.markings,arrows.meta}
\usetikzlibrary{calc}

\usepackage[pdftex,plainpages=false,colorlinks,hyperindex,bookmarksopen,linkcolor=red,citecolor=blue,urlcolor=blue]{hyperref}
\DeclareMathAlphabet{\mathpzc}{OT1}{pzc}{m}{it}

\bibpunct{[}{]}{;}{n}{,}{,}

\newtheorem{thm}{Theorem}[section]

\newtheorem{rmk}[thm]{Remark}
\newtheorem{prop}[thm]{Proposition}

\newtheorem{con}[thm]{Conjecture}
\newtheorem{coro}[thm]{Corollary}

\numberwithin{equation}{section}

\def\centerarc[#1](#2)(#3:#4:#5)
{ \draw[#1] ($(#2)+({#5*cos(#3)},{#5*sin(#3)})$) arc (#3:#4:#5); }

\allowdisplaybreaks

\newcommand{\lb}{\left (}
\newcommand{\rb}{\right )}
\newcommand{\lbb}{\left [}
\newcommand{\rbb}{\right ]}

\newcommand{\lbrb}[1]{\lb #1 \rb}

\newcommand{\langlerangle}[1]{\langle#1\rangle}
\newcommand{\lbcurly}{\left\{}
\newcommand{\rbcurly}{\right\}}

\newcommand{\curly}[1]{\lbcurly#1\rbcurly}

\newcommand{\Pbb}[1]{\Pb\lb #1\rb}
\newcommand{\Ebb}[1]{\Eb\lbb #1\rbb}

\newcommand{\limi}[1]{\lim\limits_{#1\to \infty}}

\newcommand{\limo}[1]{\lim\limits_{#1\to 0}}

\newcommand{\Eb}{\mathbb{E}}

\newcommand{\Nb}{\mathbb{N}}
\newcommand{\Rb}{\mathbb{R}}

\newcommand{\Pb}{\mathbb{P}}

\newcommand{\Ac}{\mathcal{A}}

\newcommand{\Pc}{\mathcal{P}}

\renewcommand{\Ac}{\mathtt{A}}

\renewcommand{\Phi}{\phi}
\renewcommand{\nu}{\mu}

\begin{document}
	
	\title[]{Properties and conjectures regarding discrete renewal sequences}
	\author[]{Nikolai Nikolov$^{1,2}$}
	\author[]{Mladen Savov$^{3,4}$}
		\address[]{1: Institute of Mathematics and Informatics, Bulgarian Academy of Sciences, str. Akad. Georgi Bonchev, block 8, Sofia - 1113, Bulgaria}
	\email[]{nik@math.bas.bg}
	\address[]{2: Faculty of Information Sciences,
		State University of Library Studies and Information Technologies,
		Shipchenski prohod 69A, 1574 Sofia,
		Bulgaria
	}
	\address[]{3: Faculty of Mathematics and Informatics, Sofia University ``St Kliment Ohridski'', boul. James Boucher 5, Sofia - 1164, Bulgaria}
	\email[]{msavov@fmi.uni-sofia.bg}
	\address[]{4: Institute of Mathematics and Informatics, Bulgarian Academy of Sciences, str. Akad. Georgi Bonchev, block 8, Sofia - 1113, Bulgaria}
	\email[]{mladensavov@math.bas.bg}
	

	\begin{abstract}
		In this work we review and derive some elementary properties of the discrete renewal sequences based on a positive, finite and integer-valued random variable. Our results consider these sequences as dependent on the probability masses of the underlying random variable. In particular we study the minima and the maxima of these sequences and prove that they are attained for indices of the sequences smaller or equal than the support of the underlying random variable. Noting that the minimum itself is a minimum of multi-variate polynomials we conjecture that one universal polynomial envelopes the minimum from below and that it is \textit{maximal} in some sense and \textit{largest} in another.  We prove this conjecture in a special case.
	\end{abstract}
	
	\maketitle
	\section{Introduction} Renewal theory is at the heart of classical probability theory. For this reason new results keep on appearing thereby adding further knowledge about these classical objects, see \cite{Feller2,MiOm_14} for more information. In this note we consider discrete renewal sequences, whose basic theoretical results can be found in \cite{Feller}. We review and obtain some elementary properties of the sequence of renewal masses. First, using the basic renewal equation we establish that the minimum and the maximum of the whole renewal sequence is attained within the support of the underlying probability masses. Second, we show that under natural restrictions the members of the aforementioned sequence are sandwiched between two monotone sequences which jointly converge to the limit of the well-known Blackwell's theorem. Third, given that each renewal mass is a polynomial of the probability masses we consider the minimum of the renewal sequences as a minimum of a finite number of polynomials on the respective simplex. We then introduce a universal polynomial which envelops from below this minimum and we conjecture that it is \textit{maximal} within a natural subset of polynomials. We also show that it is not \textit{largest} amongst the members of this subset. We confirm the conjecture when the renewal sequence is based on probability masses with at most three non-zero terms ( the random variable takes values only in $\curly{1,2,3}$). We show that within a subclass of polynomials the aforementioned universal polynomial is \textit{largest}. These conjectures are interesting from the standpoint of approximation theory as they involve multivariate polynomial approximation from below of a minima of a finite set of multivariate polynomials.
	\section{Notation and preliminary facts}\label{sec:not}
	In this note we work with a sequence of independent, positive, integer-valued random variables $\lbrb{X_i}_{i\geq 1}$ defined on a common probability space. Furthermore, we assume that they are mutually independent and identically distributed with finite support, that is $\Pbb{X_1=l}=p_l\geq 0, 1\leq l\leq k,$ and $\sum_{l=1}^{k}p_l=1$. Extend $p_n=0$ for $n\geq k+1$. We consider the increasing random walk $S:=\lbrb{S_n}_{n\geq 0}$ defined as $S_0=0$ and $S_n=\sum_{k=1}^n X_k,n\geq 1$, and the related to it renewal masses
	\begin{equation}\label{eq:u}
		u_n=\Pbb{n\in\curly{S_0,S_1,\cdots}}=\sum_{l=0}^\infty \Pbb{S_l=n},\,\,n\geq 0.
	\end{equation}
  Obviously $u_0=1$ and the well-known recurrent relation holds
  \begin{equation}\label{eq:rec}
  	u_n=\sum_{l=1}^n p_lu_{n-l}=\sum_{l=1}^{\min\curly{n,k}}p_lu_{n-l},
  \end{equation}
  where we have used that $p_l=0,l\geq k+1$. Blackwell's theorem states that for aperiodic random walks, that is their support is $\Nb$ or equivalently for all large $n$, $S_n$ does not live on a sub-lattice of $\Nb$,
  \begin{equation}\label{eq:lim}
  	\limi{n}u_n=\frac{1}{\Ebb{X_1}}=\frac{1}{\sum_{l=1}^k lp_l}
  \end{equation}
  see \cite[XIII Theorem 3]{Feller}.
  For example, the random walk is clearly aperiodic if $p_1>0$.
  For the purpose of our investigation we introduce
  \begin{equation}\label{def:Mm}
  	M_k=\max_{l\geq 1}\curly{ u_l} \text{ and } m_k=\min_{l\geq 1}\curly{ u_l}.
  \end{equation}
  From Proposition \ref{prop:Mm} we see that $m_k=m_k(p_1,\cdots,p_{k-1})$ and it is the minimum of $k-1$ multi-variate polynomials.
  Next, with each random walk as defined above, \textbf{we introduce for $1\leq n\leq k$, with the convention $\prod_{j=1}^0=1$,}
  \begin{equation}\label{def:poly}
  	\mathbf{Q_n\lbrb{p_1,p_2,\cdots, p_{n-1}}=\prod_{j=1}^{n-1}\sum_{l=1}^j p_l,}
  \end{equation}
  
  Note that if $\sum_{l=1}^jp_l=1, 1\leq j<k,$ and $\sum_{l=1}^{j-1}p_l<1$ then nominally $Q_k$ depends at most on $p_1,\cdots, p_{j-1}$. We note that in this context we are interested  in the behaviour of $Q_k$ on
  \begin{equation}\label{eq:reg}
  	\begin{split}
  		&A_k=\curly{\lbrb{p_1,\cdots,p_{k-1}}: p_l\geq 0, 1\leq l\leq k-1; \sum_{l=1}^{k-1}p_l\leq 1}\subseteq \Rb^{k-1}.
  	\end{split}
  \end{equation}
We set $\Pc_k$ for the set of polynomials of $k-1$ variables. We introduce partial ordering in $\Pc_k$ in the following manner: we say that $P_1\prec P_2, P_1,P_2\in \Pc_k,$  if and only if $P_1\leq P_2$ on $A_k$. This is inherited from the partial ordering of functions of $k-1$ variables on $A_k$ defined as $f\prec g\iff f\leq  g$ on $A_k$. Furthermore, we consider
\begin{equation}\label{def:Ac}
	\begin{split}
		&\Ac_k:=\curly{P\in \Pc_k: deg(P)\leq k-1,\,P\prec m_k },
	\end{split}
\end{equation}
where $deg(P)$ is the power of $P$, i.e. the highest combined power of every monomial constituting $P$.
We say that $P\in \Ac_k$ is \textit{maximal} if and only if
\begin{equation}\label{eq:max}
	\begin{split}
		& \text{$\tilde{P}\in \Ac_k$ and $\tilde{P}\succ P$ $\Rightarrow$ $\tilde{P}=P$.}
	\end{split}
\end{equation}
\textit{Largest} element is $P\in \Ac_k$ such that $P\succ Q$ for all $Q\in\Ac_k$. Furthermore, for any $P\in\Ac_k$ we set $\hat{P}(p_1,\cdots,p_{k-1})=P(p_1,p_2-p_1,\dots, p_k-p_{k-1})$ and $\hat{P}_j(p_1,\cdots,p_{j})=\hat{P}(p_1,\cdots,p_{j},1,\cdots,1),j\leq k-1$. We set
\begin{equation}\label{def:Ac1}
	\begin{split}
		&\hat{\Ac}_k=\curly{P\in\Ac_k: \deg(\hat{P}_j)\leq j\,\forall 1\leq j\leq k-1}.
	\end{split}
\end{equation}
Let us comment briefly on the classes $\Ac_k$ and $\hat{\Ac}_k$. According to Corollary \ref{cor:poly} the running minimum $m_k$ as defined in \eqref{def:Mm} is a minimum of polynomials of $p_1,\dots, p_{k-1}$. Then $\Ac_k$ is a natural choice for class polynomials that envelop from below $m_k$ uniformly in $k$. Theorem \ref{thm:poly}, however, reveals that $\Ac_k$ does not possess a maximal element, and therefore an unique element that envelops $m_k$ from below. The structure of $u_l$ suggests our choice of $\hat{A}_k$ and for $k=3$ it possesses a largest element thereby confirming our Conjecture \ref{con1}.

Here and hereafter we shall consider solely random walks $S$ and their renewal masses as defined above. Extensive account of renewal theory can be found in \cite{Feller,Feller2}.
\section{Main results and Conjectures}
The first result shows that $M_k$ ( respectively $m_k$) are attained within the first $k$ ( respectively $k-1$) of the renewal masses.
\begin{prop}\label{prop:Mm}
	Let $\lbrb{u_n}_{n\geq 0}$ be the renewal masses of
	the random walk $S$ defined above. Then
	\begin{equation}\label{eq:Mm}
		1\geq M_k=\max_{1\leq l\leq k} u_l \text{ and } m_k=\max_{1\leq l\leq k-1} u_l\geq 0.
	\end{equation}
	$M_k=1 \iff p_j=1$, for some $k\geq j\geq 1$, and $m_k=0\iff p_1=0$.
	Moreover, if $p_j\neq 1$, for all $k\geq j\geq 1$, then for $n>k$, $u_n<M_k$, and, if the walk is aperiodic $u_n>m_k$, for $n>k-1$. Next, if $p_1\in\lbrb{0,1}$  the sequence $\curly{u_n}_{n\geq 1}$ is not ultimately monotone and
	\begin{equation}\label{eq:Mm1}
		\begin{split}
			& u_n<\max_{1 \leq l\leq k}\curly{u_{n-l}}=:b_n \text{ for $n>k$};\\
			& u_n>\min_{1 \leq l\leq k}\curly{u_{n-l}}=:c_n \text{ for $n>k-1$}.
		\end{split}
	\end{equation}
	Finally, the sequences $\curly{b_n}_{n\geq k}$ and $\curly{c_n}_{n\geq k-1}$ are respectively monotone non-increasing and monotone non-decreasing.
\end{prop}
\begin{rmk}\label{rem:Mk1}
The final claim establishes that, provided $p_1\in\lbrb{0,1}$,  $u_n$ is sandwiched between two monotone sequences.
\end{rmk}
Next we formulate a simple result that is used below.
\begin{prop}\label{prop:poly}
	Let $\lbrb{u_n}_{n\geq 0}$ be the renewal masses of
	the random walk $S$ defined above. Then, for $l\geq 1$,
	\begin{equation}\label{eq:poly}
	u_l =P_l\lbrb{p_1,\cdots, p_{\min\curly{k-1,l}}},
	\end{equation}
  where $P_l,l\geq 1$ satisfy $deg(P_l)=\min\curly{k-1,l}$. For each $P_l$ each variable $p_j,j\leq l$, has a  highest possible power of precisely $l| j$ or the integer division of $l$ by $j$. Finally, for any $t>0$, and any $l\geq 1$
  	\begin{equation}\label{eq:poly'}
  	P_l\lbrb{t^{\frac{1}{l}}p_1, t^{\frac{2}{l}}p_2,\cdots, t^{\frac{\min\curly{k-1,l}}{l}}p_{\min\curly{k-1,l}}}=tP_l\lbrb{p_1,\cdots, p_{\min\curly{k-1,l}}}.
  \end{equation}
\end{prop}
We have the immediate corollary.
\begin{coro}\label{cor:poly}
	Let $\lbrb{u_n}_{n\geq 0}$ be the renewal masses of
	the random walk $S$ defined above. Then
	\begin{equation}\label{eq:poly1}
		\begin{split}
			&M_k=M_k\lbrb{p_1,p_2,\cdots, p_{k-1}}=\max_{1\leq l\leq k}\curly{P_l\lbrb{p_1,\cdots, p_{\min\curly{k-1,l}}}}\\
			&m_k=m_k\lbrb{p_1,p_2,\cdots, p_{k-1}}=\min_{1\leq l\leq k-1}\curly{P_l\lbrb{p_1,\cdots, p_{l}}}.
		\end{split}
	\end{equation}
\end{coro}
Next, we state a result which estimates $m_k$ from below with $Q_k$.
\begin{prop}\label{prop:poly1}
Let $\lbrb{u_n}_{n\geq 0}$ be the renewal masses of
the random walk $S$ defined above. Then
\begin{equation}\label{eq:bound}
	m_k\geq Q_k.
\end{equation}
Furthermore, for any $1\leq n \leq k-1$,
\begin{equation}\label{eq:bound1}
	u_n\geq Q_{n+1}.
\end{equation}
\end{prop}
We proceed to prove our main result concerning the number of maximal elements in  $\Ac_k$.
\begin{thm}\label{thm:poly}
	Let $\lbrb{u_n}_{n\geq 0}$ be the renewal masses of
	the random walk $S$ defined above. For $k\geq 3$ there is no \textit{largest} element in $\Ac_k$.
\end{thm}

Given that $\Ac_k,k\geq 3,$ has no \textit{largest} element we formulate the following conjecture.
\begin{con}\label{con1}
	For any $k\geq 3$, $Q_k$ is a maximal element in $\Ac_k$ and it is largest in $\hat{\Ac}_k$.
\end{con}
In the direction of Conjecture \ref{con1} we prove a weaker version of it.
\begin{prop}\label{prop:con}
	Conjecture \ref{con1} is valid for $k=3$.
\end{prop}
\section{Proofs}
\begin{proof}[Proof of Proposition \ref{prop:Mm}]
	Let $n>k$ and assume that $M:=M_k=u_n$.  We consider three separate cases. First, we assume that $1>p_1>0$. Applying \eqref{eq:rec} we get that
	\[M=\sum_{l=1}^{k}p_lu_{n-l}\]
	which is only possible provided $u_{n-l}=M$ for all $1\leq l\leq k$. Using this recursively we get that $u_l$ must be equal to $M$, for all $l\geq 1$. Since the random walk is aperiodic and  from \eqref{eq:lim} $M=\limi{n}u_n=\frac{1}{\Ebb{X_1}}<1$.
	 However, then from \eqref{eq:rec}
	\[M=u_k=p_k+\sum_{l=1}^{k-1}p_lu_{k-l}=p_k+M\sum_{l=1}^{k-1}p_l>M\]
	and we arrive at a contradiction. Second, we assume $p_j=1$ for some $1\leq j\leq k$ and then clearly $M=1$  and $n=jm$ for some $m\geq 1$. Third, if $p_1=0$ and no $p_j=1$ then we set $A=\curly{1\leq l\leq k: p_l>0}$ and from $M:=M_k=u_n$ for some $n>k$ we again deduce from \eqref{eq:rec} that $u_l=M$ for all $l\in A$. This easily yields a contradiction  from \eqref{eq:rec} when applied for $u_j$ with $j\in A$ and $j$ not the maximal element of $A$. Clearly, $m_k=0\iff p_1=0$ and then the second relation of \eqref{eq:Mm} holds true. Assume that $m:=m_k>0$.  Hence, $p_1>0$ and the random walk is aperiodic. Let $m:=m_k=u_n>0$, for some $n>k$. If $p_1=1$ then $M=m=1$ and the second relation of \eqref{eq:Mm} holds true. If $p_1<1$ then from  \eqref{eq:rec}  we have that
	\[m=u_k=p_k+\sum_{l=1}^{k-1}p_lu_{k-l}>\min_{1\leq l\leq k-1} u_l.\]
	This yields a contradiction and validates \eqref{eq:Mm}.  Since $0<p_1<1$ and thus $M<1$ then $u_n=\sum_{l=1}^{k}p_lu_{n-l},n\geq k,$ easily gives that $\curly{u_n}_{n\geq 1}$ can be neither monotone increasing nor monotone decreasing. Finally, it is elementary to repeat the first argument of this proof to check that it is impossible that $u_n=\max_{1 \leq l\leq k}\curly{u_{n-l}}=c$ for some for $n>k$ as otherwise we would again obtain that $u_n=c=p_1$, for all $n\geq 1$ and we would arrive at contradiction as above. The same can be done for the minimum and \eqref{eq:Mm1} is established. Let us next prove the monotonicity of $\curly{b_n}_{n\geq k}$ with $\curly{c_n}_{n\geq k-1}$ being the same. Assume that $b_{n+1}> b_n$. Then from \eqref{eq:rec} and the definition of $b_n$ we get a contradiction from
	\[b_n<b_{n+1}=u_{n}=\sum_{l=1}^{k}p_lu_{n-l}\leq b_n.\]
\end{proof}
\begin{proof}[Proof of Proposition \ref{prop:poly}]
	The relation \eqref{eq:poly} is immediate from \eqref{eq:rec}. The final claim is elementary from probabilistic standpoint: given $l\geq $ and $1\leq j\leq l$ we cannot have a piece of path that reaches $l$ with more than  $l| j$ occurrences of  value $j$ among the random variables construing it; this means that the power of $p_j$ does not exceed $l| j$; however, as $p_1>0$ then clearly we can reach $l$ by $l| j$ steps of size $j$ and then adding steps of value $1$. To show \eqref{eq:poly'} we observe that every summand in $P_l$ consists of probabilities of possible path to $l$ with the weighted sum of their powers by their index equalling $l$. Therefore, to total power contributed by the powers of $t$ is simply $1$.
\end{proof}
\begin{proof}[Proof of Proposition \ref{prop:poly1}]
	Clearly,  $1>Q_n$ is non-increasing in $n$ and it suffices to prove \eqref{eq:bound1}. We do so by induction with obvious basis $u_1=p_1\geq p_1$. Assume that \eqref{eq:bound1} holds for $n<k-1$. We have from \eqref{eq:rec}, the monotonicity of $v_n$ and the induction hypothesis that
	\[u_{n+1}=\sum_{j=1}^{n+1}p_ju_{n+1-j}\geq Q_n\sum_{j=1}^{n+1}p_j=Q_{n+1}.\]
	Thus the claim is furnished.
\end{proof}
\begin{proof}[Proof of Theorem \ref{thm:poly}]
	Assume that $P\in \Ac_k, k\geq 3$ is \textit{largest} element. We split the domain $A_k$, see \eqref{eq:reg}, as subregions whereby each $P_l$ defined in \eqref{eq:poly} is the smallest, that is $A_k\supseteq\bigcup_{l=1}^{k-1} A_k(l),$ $A_k(l)=A_k\cap \curly{P_l<P_{j},\forall j\ne l}.$ Since $P_l$ are polynomials then $A_k(l)$ is open set for each $1\leq l\leq k-1$ and the Lebesgue measure of $A_k\setminus \bigcup_{l=1}^{k-1} A_k(l)$  is zero. Choose $p_0\in A_k(l)$ and consider $\tilde{P}_l(p)=P_l(p)-a\langlerangle{p-p_0,p-p_0}, p,p_0\in A_k$ with $\langlerangle{\cdot,\cdot}$ standing for the scalar product in $\Rb^{k-1}$ and $a>0$ a scalar. Since $A_k$ is a compact then there exists $a_0>0$ such that $m_k\geq \tilde{P}_l$ on $A_k$ and therefore $\tilde{P}_l\in\Ac_k$. Also, $\tilde{P}_l(p_0)=P_l(p_0)=m_k(p_0)$ and since $P$ is the largest then $P(p_0)=\tilde{P}_l(p_0)=m_k(p_0)$. The latter can be made valid for any $p_0$ in any $A_k(l),1\leq l\leq k-1$ and hence $P=m_k$ which is impossible.
\end{proof}
\begin{proof}[Proof of Proposition \ref{prop:con}]
	When $k=3$ we have, see \eqref{eq:Mm},
	\begin{equation*}
		\begin{split}
			&m_3=\min\curly{p_1,p^2_1+p_2}=p_1\min\curly{1,\frac{p_2}{p_1}+p_1}.
		\end{split}
	\end{equation*}
	If we assume that $Q\succ Q_3, Q\in \Ac_3$ on $\Ac_3$ then clearly we obtain from the definition of $Q_3$, see \eqref{def:poly}, that
	\begin{equation*}
		\begin{split}
			&p_2\leq \limo{p_1}\frac{Q(p_1,p_2)}{p_1}\leq \limo{p_1}\min\curly{1,\frac{p_2}{p_1}+p_1}=1.
		\end{split}
	\end{equation*}
	Since $p_2$ is arbitrary, choosing it to tend to $1$ we get since $\deg Q\leq 2$, that $Q(p_1,p_2)=p_1(a+bp_1+cp_2)$. Hence on  $\Ac_3$
	\begin{equation*}
		\begin{split}
			& p_1+p_2\leq (a+bp_1+cp_2)\leq \min\curly{1,\frac{p_2}{p_1}+p_1}.
		\end{split}
	\end{equation*}
	Setting $p_2=0$ we first get that $a=0$ and then $b=1$. Substituting we get on  $\Ac_3$ the inequalities
	\begin{equation*}
		\begin{split}
			& p_2\leq cp_2\leq \min\curly{1-p_1,\frac{p_2}{p_1}}\leq \frac{p_2}{p_1}.
		\end{split}
	\end{equation*}
	Letting $p_1\to 1$ we conclude that $c=1$ and hence $Q=Q_3$. Next, let $P(p_1,p_2)=ap_1^2+bp_1p_2+cp_2^2+dp_1+ep_2+f\in \hat{\Ac}_3$. Since $P(p_1,1-p_1)$ has by assumption degree $1$, see \eqref{def:Ac1}, we conclude that $2b=a+c$. We consider $H:=P-Q_3$ and note that on $\partial A_3$, $Q_3=m_3$. Therefore, on $\partial A_3$ we have that $H\leq 0$. Note that $4\mbox{Det}(H)=4(a-1)c-(b-1)^2=-(a-c-1)^2$ and as it is well-known if $\mbox{Det}(H)\leq 0,$
	then $\max_{\overline G}H=\max_{\partial G}H$ ( and $\min_{\overline G}H=\min_{\partial G}H$)
	for each bounded region $G$ in the plane. Therefore, $P\le Q_3$ on $A_3$ and the claim that $Q_3$ is largest in  $\hat{\Ac}_3$ is established.
\end{proof}
	\section*{Acknowledgements}  M. Savov acknowledges that this study is financed by the European Union - NextGenerationEU, through the National Recovery and Resilience Plan of the Republic of Bulgaria, project No. BG-RRP-2.004-0008. 

\bibliographystyle{abbrv}

\end{document}